\documentclass[amscd,amssymb,11pt]{amsart}

\usepackage{amssymb}
\usepackage{graphicx}
\usepackage{xcolor}
\usepackage[all]{xy}

\newtheorem{thm}{Theorem}[section]
\newtheorem{lem}[thm]{Lemma}
\newtheorem{cor}[thm]{Corollary}
\newtheorem{prop}[thm]{Proposition}

\theoremstyle{definition}

\newtheorem{defns}[thm]{Definitions}
\newtheorem{quest}[thm]{Question}

\theoremstyle{remark}

\numberwithin{equation}{section}

\newcommand{\thmref}[1]{Theorem~\ref{#1}}
\newcommand{\corref}[1]{Corollary~\ref{#1}}
\newcommand{\secref}[1]{\S\ref{#1}}

\newcommand{\propref}[1]{Proposition~\ref{#1}}
\newcommand{\lemref}[1]{Lemma~\ref{#1}}

\newcommand{\Hom}{\operatorname{Hom}}

\newcommand{\Ker}{\operatorname{Ker}}

\newcommand{\Z}{{\mathbb  Z}}

\newcommand{\Q}{{\mathbb  Q}}

\newcommand{\sm}{\wedge}

\newcommand{\ra}{\rightarrow}
\newcommand{\xra}{\xrightarrow}

\newcommand{\hra}{\hookrightarrow}

\begin{document}

\title[Chromatic Smith Theorem]{An elementary proof of the chromatic \\ Smith Fixed Point Theorem}

\author[Balderrama]{William Balderrama}
\email{eqr8nm@virginia.edu}

\author[Kuhn]{Nicholas J.~Kuhn}
\email{njk4x@virginia.edu}

\address{Department of Mathematics \\ University of Virginia \\ Charlottesville, VA 22903}


\date{\today}


\begin{abstract}

A recent theorem by T. Barthel, M. Hausmann, N. Naumann, T. Nikolaus, J. Noel, and N. Stapleton says that if $A$ is a finite abelian $p$--group of rank $r$, then any finite $A$--space $X$ which is acyclic in the $n$th Morava $K$--theory with $n \geq r$ will have its subspace $X^A$ of fixed points acyclic in the $(n-r)$th Morava $K$--theory.  This is a chromatic homotopy version of P.~A.~Smith's classical theorem that if $X$ is acyclic in mod $p$ homology, then so is $X^A$.

The main purpose of this paper is to give an elementary proof of this new theorem that uses minimal background, and follows, as much as possible, the reasoning in standard proofs of the classical theorem.  We also give a new fixed point theorem for finite dimensional, but possibly infinite, $A$-CW complexes, which suggests some open problems.
\end{abstract}

\maketitle

\section{Introduction} \label{introduction}

Fixing a prime $p$ and finite group $G$, say that $G$--space $X$ is a finite $G$--space if its $p$--localization is a retract of the $p$--localization of a finite $G$--CW complex in the $G$--equivariant homotopy category.  We let $X^G$ denote its subspace of fixed points.

Let $K(n)_*$ denote the $n$th Morava $K$--theory at the prime $p$. In particular, $K(1)_*$ is a summand of complex $K$--theory with mod $p$ coefficients, and $K(0)_*$ is rational homology.

A key result in \cite{6 author} can be stated as follows.

\begin{thm} \label{main thm}  Let $A$ be a finite abelian $p$-group of rank $r$, and let $X$ be a finite $A$--space. If $\widetilde K(n)_*(X)=0$ with $n \geq r$, then $\widetilde K(n-r)_*(X^{A})=0$.
\end{thm}

This is a chromatic homotopy theory analogue of the following classical theorem of P.~A.~Smith \cite{PA Smith 1941}.

\begin{thm} \label{classical Smith thm} Let $P$ be a finite $p$-group, and let $X$ be a finite dimensional $P$--space. If $\widetilde H_*(X;\Z/p) = 0$, then $\widetilde H_*(X^{P};\Z/p) = 0$.
\end{thm}

We note that \thmref{main thm} follows by iteration from the cases when $A = C_{p^k}$, the cyclic $p$--group of order $p^k$, and \thmref{classical Smith thm} follows from the case when $P=C_p$.

The main purpose of this note is to give an elementary proof of \thmref{main thm} that follows, as much as possible, the reasoning in standard proofs of the classical theorem. All needed background material about $K(n)_*$ and related theories is in papers published by 2000.  We hope that our presentation will lend some clarity about the interesting remaining problems in this area.

The classical theorem includes the statement that $X^P \neq \emptyset$, while the version of \thmref{main thm} proved in \cite{6 author} implicitly assumes that $X$ has a point fixed by $A$, an assumption we do not need to make. Indeed, the first steps in our proof hold when $X$ is just assumed to be finite dimensional, and not necessarily finite, and lead to the following new fixed point theorem.

\begin{thm} \label{fixed point thm}  Let $A$ be a finite abelian $p$--group of rank $r$, and $X$ a finite dimensional $A$--CW complex.  If $\widetilde K(n)_*(X)=0$ with $n \geq r$, then $X^{A} \neq \emptyset$.
\end{thm}

We note that, if $X$ is any space, then $\widetilde K(n)_*(X)=0 \Rightarrow\widetilde K(r)_*(X)=0$ for $n \geq r \geq 1$ \cite{bousfield K(n) equivs}. (This generalized Ravenel's result \cite[Thm.2.11]{ravenel conjs paper} about finite $X$.) Thus this fixed point theorem for all $n$ is implied by the special case when $n=r$.

In \secref{proof section 1}, we recall some needed background material, and deduce some simple consequences. The theorems are then quickly proved in \secref{proof section 2}.

A final section has various remarks and speculations. In particular, we wonder if some weakening of the finiteness hypothesis of \thmref{main thm} might be possible, and we are curious about the existence of examples showing that \thmref{fixed point thm} is as strong as possible.

\subsection{Acknowledgements} The second author is a PI of RTG NSF grant DMS-1839968, which has supported the research of the first author. Some of the writing of this paper was done while the second author was visiting the Utrecht Geometry Center, with support from Simons Foundation Collaboration Grant 709802.  We thank Neil Strickland for helpful email about literature references, and Tom Goodwillie and Ian Leary for their helpful answers to a Mathoverflow query (see \secref{infinite complexes questions}).

\section{Background material and a localization result} \label{proof section 1}

\subsection{Morava $E$--theory and $E_n^*(BA)$}

Recall the Brown--Peterson homology theory $BP$, with coefficient ring $BP_\ast = \Z_{(p)}[v_1,v_2,\ldots]$.  We work with complex oriented theories $E$ that are $p$--local and with $p$-typical formal group laws. The coefficient ring $E_\ast$ of such a theory is a $BP_\ast$-algebra, and is said to be Landweber exact and $v_n$-periodic if $v_0,v_1,\ldots$ acts as a regular sequence on $E_\ast$ with $v_n$ acting as a unit on $E_\ast/(v_0,\ldots,v_{n-1})$. (Throughout, we let $v_0 = p$, as is standard.)

\begin{lem} \label{homology lemma}  \cite[Cor. 1.12]{hovey splitting conj} Let $E_*$ be Landweber exact and $v_n$--periodic. A spectrum is $E_*$--acyclic if and only if it is $K(m)_*$--acyclic for $0\leq m \leq n$. Thus, if a finite spectrum is $E^*$--acyclic then it is $K(n)_*$--acyclic.
\end{lem}

Now let $E_n$ be the $n$th Morava $E$--theory, as in \cite{hovey cohom classes}, \cite{hoveystrickland}, or \cite{hkr}. There are variants of these, so for concreteness, we will say that
$$E_n^* = \Z_p[u^{\pm 1}][[v_1,\ldots,v_{n-1}]],$$
where $u\in E_n^{-2}$, and with complex orientation $y\in E_n^2(BU(1))$ whose associated formal group law $F$ has $p$--series of the form
$$[p](y) = py +_F v_1y^p +_F \cdots +_F v_{n-1} y^{p^{n-1}} +_F u^{p^n-1} y^{p^n}.$$
In particular, $v_n = u^{p^n-1}$. This is a Landweber exact and $v_n$-periodic theory, with the following additional property.

\begin{lem} \cite[Prop.3.6]{hovey cohom classes} \label{k acyclics = e acyclics lem}  If $f: X \ra Y$ is a map of spectra, then $f$ is an $E_n^*$--isomorphism if and only if it is a $K(n)_*$--isomorphism.
\end{lem}

\begin{prop}  \label{Borel acyclic prop} Let $G$ be a finite group and $X$ a $G$--space. If $\widetilde K(n)_*(X) = 0$, then the map $EG \times_G X \ra BG$ induces an isomorphism
$$E_n^*(BG) \simeq E_n^*(EG \times_G X).$$
\end{prop}
\begin{proof} There are implications \begin{equation*}
\begin{split}
\widetilde K(n)_*(X)=0 & \Leftrightarrow
 K(n)_*(X) \simeq K(n)_*(pt)\\
  & \Rightarrow K(n)_*(EG\times_G X) \simeq K(n)_*(BG)  \\
  & \Leftrightarrow E_n^*(BG) \simeq E_n^*(EG\times_G X).
\end{split}
\end{equation*}
The previous lemma gives the third implication. The second implication is a special case of a general fact: if $f: X \ra Y$ is a $G$--equivariant map between $G$--spaces that is an isomorphism in a generalized homology theory $E_*$, then it induces an isomorphism on the associated Borel theory  $E_*(EG \times_G X) \xra{\sim} E_*(EG \times_G Y)$.
\end{proof}

We recall some basic calculations as in \cite[\S 5]{hkr}.  The complex orientation $y \in E_n^2(BU(1))$ defines an isomorphism $E_n^*(BU(1)) \simeq E_n^*[[y]]$, and the $p^k$--series satisfies $[p^k](y) \equiv u^{p^{nk}-1}y^{p^{nk}} \mod (v_0,\ldots,v_{n-1},y^{p^{nk}+1})$.

The standard inclusion $i_k: C_{p^k} \hra U(1)$ induces an isomorphism
$$E_n^*(BC_{p^k}) \simeq  E_n^*[[y_k]]/([p^k](y_k)),$$
where $y_k = i_k^*(y)$, and it follows from the Weierstrass preparation theorem that $E_n^*(BC_{p^k})$ is a free $E_n^*$--module with basis ${1,y_k,y_k^2, \dots, y_k^{p^{nk}-1}}$.

Similarly, if $l<k$ and $q: C_{p^k} \ra C_{p^l}$ is the standard quotient map, then the map
$q^*: E_n^*(BC_{p^l}) \ra E_n^*(BC_{p^k})$ satisfies $q^*(y_l) = [p^{k-l}](y_k)$, and induces an isomorphism
$$ E_n^*(BC_{p^k}) \simeq E_n^*(BC_{p^l})[[y_k]]/([p^{k-l}](y_k)-y_l).$$
This makes $E_n^*(BC_{p^k})$ into a free $E_n^*(BC_{p^l})$ module of rank $p^{n(k-l)}$.

From these calculations a couple of general facts follow.

\begin{lem} \label{E(BA) lemma} {\bf (a)} \  If $A$ is a finite abelian $p$-group, then $E_n^*(BA)$ is a free $E_n^*$--module of rank $|A|^n$, and there is a natural isomorphism
$$ E_n^*(BA) \otimes_{E_n^*} E_n^*(X) \simeq E_n^*(BA \times X)$$
for any space $X$.

\noindent{\bf (b)} \  If $q: A \ra \bar A$ is an epimorphism between two finite abelian $p$-groups, then $q^*: E_n^*(B\bar A) \ra E^*_n(BA)$ makes $E^*_n(BA)$ into a free $E^*_n(B\bar A)$--module of rank $\lvert\Ker q\rvert^n$.
\end{lem}

\subsection{Localization} \label{localizations}

Let $A$ be a finite abelian $p$-group. Observe that the group $C_p^\times$ of automorphisms of $C_p$ acts freely on the set of nontrivial homomorphisms  $\alpha: A \ra C_p$ via postcomposition.

If $\alpha: A \ra C_p$ is a homomorphism, we let $e(\alpha) = \alpha^*(y_1) \in E^2_n(BA)$. This is the Euler class of the 1--dimensional complex representation defined by the composite $A \xra{\alpha} C_p \hra U(1)$.  With this notation, we define two Euler classes $e(A), \bar e(A) \in E_n^*(BA)$ as follows.

\begin{defns} Let $e(A) = \prod e(\alpha)$, with the product over all nontrivial $\alpha: A \ra C_p$.  Let $\bar e(A) = \prod e(\alpha)$, with the product over one representative of each $C_p^{\times}$--orbit. The definition of $\bar e(A)$ involves choices, but, for concreteness, we let $\bar e(C_p) = y_1$, i.e.\ we choose the identity $C_p \ra C_p$.
\end{defns}

Basic facts about these are summarized in the next lemma.

\begin{lem} \label{Euler class lemma} {\bf (a)} \ If $i: A^{\prime} \hra A$ is the inclusion of a proper subgroup, then $i^*(e(A)) = i^*(\bar e(A)) = 0$.

\noindent{\bf (b)} \ $e(A)^{-1}E_n^*(BA) = \bar e(A)^{-1}E_n^*(BA)$.
\end{lem}
\begin{proof} $e(A)$ and $\bar e(A)$ are products of the various $e(\alpha)$. If $A^{\prime}$ is a proper subgroup of $A$, then at least one $\alpha$ restricts to the trivial representation of $A^{\prime}$. As the Euler class of the trivial representation vanishes, statement (a) follows.

To prove statement (b), we first note that, since $\bar e(A)$ divides $e(A)$, inverting $e(A)$ also inverts $\bar e(A)$.  To see the converse, we need to show that, given a nontrivial $\alpha: A \ra C_p$, inverting $e(\alpha)$ also inverts $e(\beta)$ for all $\beta$ in the $C_p^{\times}$--orbit of $\alpha$.  Such $\beta$ are given by the composites $A \xra{\alpha} C_p \xra{m} C_p$, with $m\in\{1, \dots, p-1\}$.

Since $e(m\circ \alpha) = [m](e(\alpha))$, it suffices to show that inverting $e(\alpha)$ also inverts $[m](e(\alpha))$ for any $m\in\{1,\ldots,p-1\}$.  One way to see this goes as follows.  Choose $s$ such that $sm \equiv 1 \mod p$. Then $y_1 = [s]([m](y_1)) \in E_n^*(BC_p)$, so that $e(\alpha) = [s]([m](e(\alpha))) \in E_n^*(BA)$.  Since $x$ always divides $[s](x)$, we see that $[m](e(\alpha))$ divides $e(\alpha)$, and so inverting $e(\alpha)$ also inverts $[m](e(\alpha))$.
\end{proof}

\begin{prop} \label{localization prop}  Let $A$ be a finite abelian $p$--group, and let $e$ be either $e(A)$ or $\bar e(A)$.  If $X$ is a finite dimensional $A$--CW complex, then the inclusion $ X^A \hra X$ induces an isomorphism
$$ e^{-1}E_n^*(EA\times_A X) \xra{\sim} e^{-1}E_n^*(BA \times X^A).$$
\end{prop}
\begin{proof}  This is an application of classic localization theory, as in \cite[Chapter III]{tom Dieck}, and a simple proof goes as follows. For notational simplicity, let $F = X^{A}$. The $A$--equivariant cofibration sequence $F \ra X \ra X/F$ induces a cofibration sequence
$$ BA \times F \ra EA \times_A X \ra EA_+ \sm_A X/F,$$
so we need to show that $e^{-1}\widetilde E_{n}^*(EA_+\sm_{A} X/F ) = 0$.  Since $X$ is finite dimensional, $X/F$ has a finite filtration by its equivariant skeleta $(X/F)^j$, so it suffices to show $e^{-1}\widetilde E_{n}^*(EA_+\sm_{A} (X/F)^j/(X/F)^{j-1} ) = 0$ for each $j$.  One has an equivariant equivalence of the form
$$(X/F)^j/(X/F)^{j-1} \simeq \bigvee_i \Sigma^j (A/A_i)_+$$
where each $A_i$ is a {\em proper} subgroup of $A$.  Since $EA_+ \sm_A (A/A_i)_+ \simeq BA_{i+}$, we need to show that $\displaystyle e^{-1} \prod_i E_{n}^*(BA_i) = 0$.  But this follows immediately from \lemref{Euler class lemma}(a).
\end{proof}

\section{Proofs of the theorems} \label{proof section 2}

\propref{Borel acyclic prop}, \propref{localization prop}, and \lemref{E(BA) lemma} combine to prove the following theorem, which gets us much of the way towards the proofs of \thmref{main thm} and \thmref{fixed point thm}.

\begin{thm} \label{part 1 thm}  Let $A$ be a finite abelian $p$--group, and let $X$ be a finite dimensional $A$--CW complex.  If $\widetilde K(n)_*(X)=0$, then, with $e$ either $e(A)$ or $\bar e(A)$,  the map $X^A \ra pt$ induces an isomorphism
$$e^{-1}E_n^*(BA) \xra{\sim} e^{-1}E_n^*(BA) \otimes_{E_n^*} E_n^*(X^A).$$
\end{thm}

Now we need to know something about $e^{-1}E_n^*(BA)$.
The following algebraic result of Mark Hovey and Hal Sadofsky will suffice to prove \thmref{main thm}.

\begin{prop} \label{hovey sadofsky prop} Let $e = \bar e(C_p) = y \in E_n^*(BC_p) \simeq E_n^*[[y]]/([p](y))$. The ring $e^{-1}E_n^*(BC_p) \simeq E_n^*((y))/([p](y))$ is Landweber exact and $v_{n-1}$-periodic.
\end{prop}
\begin{proof}
This is given an elementary proof in \cite[p.3583]{hovey sadofsky}. For completeness, and to emphasize its elementary nature, we give a slightly different short proof.

As $E_n^* (BC_p)$ is a free $E_n^*$-module, and regular sequences are preserved by localization, the ring $e^{-1}E_n^*(BC_p)$ is Landweber exact. We must only show that it is $v_{n-1}$-periodic. Let $I = (v_0,v_1,\ldots,v_{n-2})$. Then we must show that $(e^{-1}E_n^*(BC_p))/I$ is nonzero and that $v_{n-1}$ is a unit in this ring.

Recall $v_n = u^{p^n-1}$. We have $[p](y) \equiv v_{n-1}y^{p^{n-1}} +_F v_ny^{p^n} \mod I$, so that
$$ v_{n-1}y^{p^{n-1}} \equiv [-1](v_ny^{p^n}) \mod (I,[p](y))$$
where $[-1](y)$ is the $(-1)$--series associated to the formal group law $F$.

When $n \geq 2$, it follows that
\begin{align*}
(e^{-1}E_n^*(BC_p))/I &\simeq \Z/p[u^{\pm 1}][[v_{n-1}]]((y))/(v_{n-1}y^{p^{n-1}}+_F v_ny^{p^n}) \\
&\simeq \Z/p[u^{\pm 1}][[v_{n-1}]]((y))/(v_{n-1} - y^{-p^{n-1}} \cdot [-1] (v_ny^{p^n}))\\
& \simeq \Z/p[u^{\pm 1}]((y))
\end{align*}
since $y^{-p^{n-1}} \cdot [-1] (v_ny^{p^n})) = -v_ny^{p^n-p^{n-1}}\epsilon(y)$, where $\epsilon(y)$ is a monic power series in $y$. A similar computation applies when $n=1$, only one ends up with
$e^{-1}E_1^*(BC_p) \simeq \Z_p[u^{\pm 1}]((y))/(p-y^{-1}\cdot [-1](v_1y^p))$,
which will be a free $\Q_p[u^{\pm 1}]$--module with basis $1,y,\dots,y^{p-2}$.
In either case, this ring is visibly nonzero and contains $v_{n-1}$ as a unit as claimed.
\end{proof}

A variant of \propref{hovey sadofsky prop} will suffice to prove \thmref{fixed point thm}.

\begin{prop} \label{HKR prop} Let $n \geq r$.  The ring $e(C_p^r)^{-1}E_n^*(BC_p^r)$ is not zero.
\end{prop}
\begin{proof} The special case when $r=n$ was analyzed in \cite[\S6.2]{hkr}: the ring $e(C_p^n)^{-1}E_n^*(BC_p^n)$ is the nonzero ring called $L_1(E_*)$ there, and is in fact easily shown to be finite and faithfully flat over $p^{-1}E_n^*$.  For the general case, note that the projection $C_p^n \ra C_p^r$ onto the first $r$ coordinates induces an algebra map from $e(C_p^r)^{-1}E_n^*(BC_p^r)$ to the nonzero ring $e(C_p^n)^{-1}E_n^*(BC_p^n)$.
\end{proof}

Both of the propositions imply similar results for more general abelian groups.

\begin{cor} \label{vn cor} \ {\bf(a)} The ring $e(C_{p^k})^{-1}E_n^*(BC_{p^k})$ is Landweber exact and $v_{n-1}$-periodic.

\noindent{\bf (b)} \ Let $A$ be a finite abelian $p$-group of rank $r$.  Let $n \geq r$ and let $e \in E_n^*(BA)$ be $e(A)$ or $\bar e(A)$.  Then the ring $e^{-1}E_n^*(BA)$ is nonzero.
\end{cor}
\begin{proof} If $A$ has rank $r$, then any surjection $q: A \ra C_p^r$ induces a bijection $q^*: \Hom(C_p^r, C_p) \xra{\sim} \Hom(A,C_p)$. It follows that $q^*(e(C_p^r)) = e(A)$.  \lemref{E(BA) lemma}(b) tells us that, via $q^*$, $E_n^*(BA)$ is a finitely generated free $E_n^*(BC_p^r)$--algebra, so $e(A)^{-1}E_n^*(BA)$ will be a finitely generated free $e(C_p^r)^{-1}E_n^*(BC_p^r)$--algebra.  Thus \propref{hovey sadofsky prop} proves statement (a) and \propref{HKR prop} implies statement (b).
\end{proof}

\subsection{Proof of \thmref{fixed point thm}}Let $A$ have rank $r$.  \thmref{part 1 thm} tells us that if $\widetilde K(n)_*(X) = 0$, then $e^{-1}E_n^*(BA) \xra{\sim} e^{-1}E_n^*(BA) \otimes_{E_n^*} E_n^*(X^A)$. \corref{vn cor}(b) tells us that, if $n \geq r$, then $e^{-1}E_n^*(BA) \neq 0$. (When $r=1$, \corref{vn cor}(a) already shows this.) Thus $E_n^*(X^A) \neq 0$, and so $X^A \neq \emptyset$.

\subsection{Proof of \thmref{main thm}}  We first observe that if $A$ is abelian and $A^{\prime}$ is a subgroup of $A$, then $X^A = (X^{A^{\prime}})^{A/A^{\prime}}$. Thus it suffices to prove \thmref{main thm} when $A$ is cyclic, as the general case will follow by iteration.

So let $C=C_{p^k}$, and let $X$ be a finite $C$-space such that $\widetilde K(n)_*(X)= 0$ for some $n \geq 1$. As before, \thmref{part 1 thm} tells us that $X^C \neq \emptyset$ and then that $e(C)^{-1}E_n^*(BC) \otimes_{E_n^*} \widetilde E_n^*(X^C) = 0$.

Given a finite CW complex $Y$, let $h^*(Y) = e(C)^{-1}E_n^*(BC) \otimes_{E_n^*} \widetilde E_n^*(Y)$. By \corref{vn cor}(a), this defines a reduced cohomology theory with coefficient ring that is Landweber exact and $v_{n-1}$--periodic. Since $h^*(X^C) = 0$, \lemref{homology lemma} tells us that $\widetilde K(n-1)_*(X^C) = 0$.

\section{Further Remarks}

\subsection{Comparison with other proofs}

The first thing to say is that every proof we know of \thmref{main thm} involves inverting Euler classes at some point in the argument.

In particular, the proof of this theorem in the special case when $A=C_p$ by Strickland \cite[Thm.16.9]{strickland} (unpublished) or Balmer and Sanders  \cite{balmer sanders} uses results of \cite{hovey sadofsky}, and Barthel et al.\ \cite{6 author} invert Euler classes in the key section 3 of their paper. (In all these papers, \thmref{main thm} is stated in terms of the geometric fixed point functor $\Phi^{A}$ applied to a compact object in $A$--spectra, but this is easily seen to be equivalent to \thmref{main thm} as stated above, only with the added assumption that there exists a fixed point.)

Unlike those proofs, our proof here switches to $E_n$--{\em co}homology, which leads to our \thmref{part 1 thm}, and also allows for our simple deduction of \thmref{main thm} for all $A$ from the case when $A=C_p$ using \corref{vn cor}, which itself has an easy proof.

Stronger than we need for our purposes is the statement that, if $A$ is a finite abelian $p$--group of rank $r$, and $n\geq r$, then $e(A)^{-1}E_n^*(BA)$ is Landweber exact and $v_{n-r}$ periodic.  In \cite{torii}, T.~Torii shows this when $A$ is elementary abelian. In \cite[Prop.5.28]{mnn}, A.~Mathew, N.~Naumann, and J.~Noel show this for general $A$, as an application of Greenlees and Strickland's analysis of the rings $E_n^*(BA)/(v_0,\ldots,v_t)$ in \cite{greenlees strickland}.

Regarding our fixed point theorem, \thmref{fixed point thm}, we note that application of the main theorem in \cite{kuhn lloyd 2020} leads to the deduction of a generalization of \thmref{fixed point thm} for all finite $p$--groups $P$ \cite[Thm.2.20]{kuhn lloyd 2020}, but specialized to the finite $P$--space case. Our argument here proving \thmref{fixed point thm} for abelian groups is simpler, and, of course, applies to all finite dimensional complexes with appropriate $A$--actions.

In a different direction, the second author has noted \cite{kuhn char proof} that \thmref{main thm} in the case when $n=r$ can be immediately `read off of' the generalized character theory of \cite{hkr}, and that the theorem in the general case can be similarly deduced from Stapleton's more general transchromatic characters \cite{stapleton}.  (This argument yields the full unbased version of \thmref{main thm}.)  Constructing these characters also involves inverting appropriate classes in the Morava $E$--theory of abelian $p$--groups (and then assembling the localized rings into an appropriate universal ring).  For the purposes of proving \thmref{main thm}, our proof here uses much less analysis of $E_n^*(BA)$ than is used in \cite{hkr} and \cite{stapleton}.

\subsection{Generalizations to non-abelian groups}

The paper \cite{balmer sanders} reduced the problem of understanding the topology of the Balmer spectrum of a stable equivariant homotopy category to a problem that can be posed as follows:  given $n \geq 0$, and a subgroup $Q$ of a finite $p$--group $P$, compute $r_n(P,Q)$, where $r_n(P,Q)$ is the minimal $r$ such that if $X$ is a finite $P$--space and $X^Q$ is $K(n+r)_*$-acyclic, then $X^P$ is $K(n)_*$-acyclic.

As discussed in \cite{kuhn lloyd 2020}, iterated use of \thmref{main thm} shows that $r_n(P,Q) \leq r(P,Q)$, where $r(P,Q)$ is the minimal $r$ such that there exists a sequence of subgroups $ Q = K_0 \lhd K_1 \lhd \dots \lhd K_r = P$ with each $K_{i-1}$ normal in $K_i$ and each $K_i/K_{i-1}$ cyclic. (In particular, $r(P,Q)$ is the rank of $P/Q$ when $P$ is abelian.)

One might ask if this upper bound is best possible.  To show this for a particular $n$ and pair $Q < P$, one needs to find a finite $P$--space $X$ such that
$$\widetilde K(n+r)_*(X^Q)=0 \text{ but } \widetilde K(n)_*(X^P)\neq 0,$$
with $r=r(P,Q)-1$. The authors of \cite{6 author} find such examples when $P$ is abelian.

The main theorem of \cite{kuhn lloyd 2020} says that the statement ``for all finite $P$-spaces $X$, if $X^Q$ is $K(n+r)_*$--acyclic then $X^P$ is $K(n)_*$-acyclic'' implies the apparently stronger statement ``for all finite $P$-spaces $X$, there is an inequality
$ \dim_{K(n+r)_*} K(n+r)_*(X^Q) \geq \dim_{K(n)_*} K(n)_*(X^P)$.''
(This sort of conclusion is analogous to a theorem of Floyd \cite{floyd TAMS 52} in the classical situation.) This leads both to interesting applications of \thmref{main thm} \cite{kuhn lloyd 2021} and to more families of examples showing that $r_n(P,Q) = r(P,Q)$ \cite{kuhn lloyd 2020}.

Some steps in our proof of \thmref{main thm} are quite formal, and can be generalized to nonabelian groups.  But the proof ultimately hinges on \corref{vn cor}, which fails in the strongest possible way for nonabelian groups: one can define an Euler class $e\in E_n^*(BG)$ for any finite group $G$, but $e^{-1}E_n^*(BG) = 0$ whenever $G$ is nonabelian. If there are pairs $Q<P$ for which $r_n(P,Q)$ is strictly less than $r(P,Q)$, then some clever new ideas will be needed to prove this.  \cite{kuhn lloyd 2020} describes a number of pairs $Q<P$ for which $r_n(P,Q)$ is not yet known.  Here we will just advertise one of these: let $C$ be a noncentral subgroup of order 2 in the dihedral group $D$ of order 16. Is $r_n(D,C)$ equal to 2 or 3? (We note that $r(D,C) = 3$ and $r_n(D,\{e\}) = r(D,\{e\}) = 2$.)

\subsection{Questions about finite dimensional complexes} \label{infinite complexes questions}

An obvious question is whether \thmref{main thm} might hold for finite dimensional complexes that aren't necessarily finite.  Finiteness appears in our proof to ensure that $e^{-1}E_n^*(BC)\otimes_{E_n^*}E_n^*(X^C) = 0 \implies K(n-1)_* (X^C) = 0$.  Still one can ask:

\begin{quest}  Let $A$ be a finite abelian $p$-group of rank $r$, and let $X$ be a finite dimensional $A$--CW complex. If $\widetilde K(n)_*(X)=0$ with $n \geq r$, must $\widetilde K(n-r)_*(X^{A})=0$?
\end{quest}

The first thing to say is that the answer is no, if $n=r$.  If $M(\Q)$ is the mapping telescope of $S^1 \xra{2} S^1 \xra{3} S^1 \xra{4} \dots$, then $\widetilde H_*(M(\Q);\Z) = \Q$, concentrated in dimension 1. Thus if we let an abelian $p$--group $A$ act trivially on $M(\Q)$, then $\widetilde K(n)_*(M(\Q))=0$ for all $n>0$ (and all $p$), but $\widetilde K(0)_*(M(\Q)^{A}) \neq 0$. But this is really a non-equivariant example, and Bousfield's result in \cite{bousfield K(n) equivs}, that if $X$ is a space then $\widetilde{K}(n)_* (X) = 0$ implies $\widetilde{K}(n-r)_* (X) = 0$ for $n-r \geq 1$, makes it plausible that the question could have a positive answer for $n-r\geq 1$.  We note that \cite{kuhn lloyd 2020} shows that such a chromatic Smith theorem would imply that the analogous chromatic Floyd theorem would hold, and, furthermore, the fixed point theorem \cite[Thm.2.20]{kuhn lloyd 2020} would generalize to finite dimensional complexes with an action of a finite $p$--group. \\

One might also wonder if our fixed point theorem \thmref{fixed point thm} is best possible.

\begin{quest}  For all $r \geq 1$, does there exist a finite dimensional $C_p^r$--space $X$ that is $K(r-1)_*$--acyclic, and has no fixed points?
\end{quest}

The answer is yes, when $r=1$. The second author asked on Mathoverflow if $C_p$ could act on a rationally trivial finite dimensional complex without fixed points.  As answers, Tom Goodwillie \cite{goodwillie} described a free action of $C_2$ on a rationally acyclic $2$--dimensional complex, and Ian Leary pointed to a paper of his that constructs fixed point free actions of any finite group on a $3$--dimensional rationally acyclic complex \cite[Thm.13.1]{leary}.

\end{document}